\documentclass[11pt]{article}

\usepackage[margin=1in,footskip=0.25in]{geometry}

\usepackage{amsmath,amssymb,amsthm,amsfonts}

\usepackage{graphpap}
\usepackage{graphics}
\usepackage{epsfig}
\usepackage{pstricks}

\usepackage[sf, bf, small]{titlesec}
\usepackage{graphicx, psfrag}
\usepackage{latexsym}

\theoremstyle{plain}
\newtheorem{Thm}{Theorem}[section]
\newtheorem{Lem}[Thm]{Lemma}
\newtheorem{Prop}[Thm]{Proposition}
\newtheorem{Cor}[Thm]{Corollary}

\theoremstyle{definition}

\newtheorem{Rem}[Thm]{Remark}
\newtheorem{Exam}[Thm]{Example}

\title{On the phylogeny graphs of degree-bounded digraphs}

\author{
\textsc{Seung Chul LEE}
\footnote{Department of Mathematics Education,
Seoul National University, Seoul 151-742, Korea.
\textit{E-mail}: \texttt{mobum83@nate.com}}
\and
\textsc{Jihoon CHOI}
\footnote{Department of Mathematics Education,
Seoul National University, Seoul 151-742, Korea.
\textit{E-mail}: \texttt{gaouls@snu.ac.kr}}
\and
\textsc{Suh-Ryung KIM}
\footnote{Department of Mathematics Education,
Seoul National University, Seoul 151-742, Korea.
\textit{E-mail}: \texttt{srkim@snu.ac.kr}}
\and
\textsc{Yoshio SANO}
\footnote{Division of Information Engineering, Faculty of Engineering, Information and Systems,
University of Tsukuba, Ibaraki 305-8573, Japan.
\textit{E-mail}: \texttt{sano@cs.tsukuba.ac.jp}}}

\date{}

\begin{document}

\maketitle


\begin{abstract}
Hefner \textit{et al.} [K. A. S. Hefner, K. F. Jones, S. -R. Kim, R. J. Lundgren and F. S. Roberts: $(i,j)$ competition graphs,
 \emph{Discrete Applied Mathematics} \textbf{32} (1991) 241--262]
characterized acyclic digraphs each vertex of which has inderee and outdegree at most two and whose competition graphs are interval.
They called acyclic digraphs each vertex of which has inderee and outdegree at most two $(2,2)$ digraphs.
In this paper, we study the phylogeny graphs of $(2,2)$ digraphs.
Especially, we give a sufficient condition and necessary conditions for $(2,2)$ digraphs having chordal phylogeny graphs.
Phylogeny graphs are also called moral graphs in Bayesian network theory.
Our work is motivated by problems related to evidence propagation in a Bayesian network for which it is useful to know which acyclic digraphs have their moral graphs being chordal.
\end{abstract}




\section{Introduction}

Through this paper, we deal with simple graphs and simple digraphs.
In a digraph, we sometimes represent an arc $(u,v)$ by $u \rightarrow v$.

Given an acyclic digraph $D$, the \emph{competition graph} of $D$, denoted by $C(D)$, is the graph having vertex set $V(D)$ and edge set $\{ uv \mid (u,w),  (v,w) \in A(D) \text{ for some } w \in V(D)\}$.
A graph $G$ is called an \emph{interval graph} if we can assign to each vertex $x$ of $G$ a real interval $J(x)$ so that, whenever $x\neq y$, $xy \in E(G)$ if and only if $J(x) \cap J(y) \neq \emptyset$.
Cohen~\cite{Cohen} introduced the notion of competition graphs in the study on predator-prey concepts in ecological food webs.
Cohen's empirical observation that real-world competition graphs are usually interval graphs
had led to a great deal of research on the structure of competition graphs
and on the relationship between the structure of digraphs and their corresponding competition graphs.
In the same vein, various variants of competition graphs have been introduced and studied.
For recent work related to competition graphs, see~\cite{FM, KA, Kuhl, LiChang, ZR}.

Steif~\cite{Steif} showed that it might be difficult
to find the structural properties of acyclic digraphs whose competition graphs are interval.
In that respect, Hefner~{\it et al.}~\cite{Kim} placed restrictions on the indegree and the outdegree of vertices of acyclic digraphs
to obtain the list of forbidden subdigraphs for acyclic digraphs whose competition graphs are interval.

The notion of phylogeny graphs was introduced by Roberts and Sheng~\cite{RobertsSheng}
as a variant of competition graphs.
(See also \cite{Ha05, PS13, RS98, RS99, RS00, ZH06} for study on phylogeny graphs.)
Given an acyclic digraph $D$, the \emph{underlying graph} of $D$, denoted by $U(D)$,
is the graph with vertex set $V(D)$ and edge set $\{ xy \mid (x,y) \in A(D) \text{ or } (y,x) \in A(D) \}$.
The \emph{phylogeny graph} of an acyclic digraph $D$, denoted by $P(D)$, is
the graph with vertex set $V(D)$ and edge set $E(U(D)) \cup E(C(D))$.
For example, given an acyclic digraph $D$ in Figure~\ref{fig:eg31}(a),
the competition graph of $D$ is the graph $C(D)$ in Figure~\ref{fig:eg31}(b),
and the phylogeny graph of $D$ is the graph $P(D)$ in Figure~\ref{fig:eg31}(c).

\begin{figure}[h]
\psfrag{a}{(a)$~D$}
\psfrag{b}{(c)$ $}
\psfrag{c}{(c)$~P(D)$ }
\psfrag{d}{(b)$ ~C(D)$}
\psfrag{1}{$v_1$}
\psfrag{2}{$v_2$}
\psfrag{3}{$v_3$}
\psfrag{4}{$v_4$}
\psfrag{5}{$v_5$}
\psfrag{6}{$v_6$}
\psfrag{7}{$v_7$}
\psfrag{8}{$v_8$}
\psfrag{9}{$v_9$}
\centering{
\includegraphics[height=3.4cm]{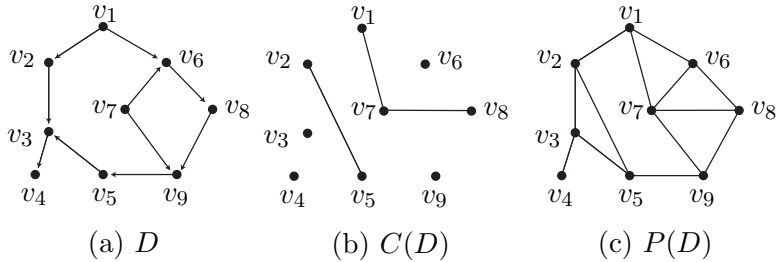}}
\caption{(a) An acyclic digraph $D$,
(b) The competition graph $C(D)$ of $D$,
(c) The phylogeny graph $P(D)$ of $D$}
\label{fig:eg31}
\end{figure}

``Moral graphs" having arisen from studying Bayesian networks are the same as phylogeny graphs.
One of the best-known problems, in the context of Bayesian networks,
is related to the propagation of evidence.
It consists of the assignment of probabilities to the values of the rest of the variables,
once the values of some variables are known.
Cooper~\cite{Cooper} showed that this problem is NP-hard.
Most noteworthy algorithms for this problem are given by
Pearl~\cite{Pearl}, Shachter~\cite{Shachter} and by Lauritzen and Spiegelhalter~\cite{Lauritzen}.
Those algorithms include a step of triangulating a moral graph,
that is, adding edges to a moral graph to form a chordal graph.

A graph $G$ is said to be \emph{chordal} if every cycle in $G$ of length greater than $3$ has a chord, namely, an edge joining two nonconsecutive vertices on the cycle, that is, $G$ does not contain a cycle of length at least $4$ as an induced subgraph.
A necessary and sufficient condition for a graph being interval is that the graph does not contain a cycle of length at least $4$ as an induced subgraph and the complement of the graph is transitively orientable.
This implies that an interval graph is chordal.

As triangulations of moral graphs play an important role in algorithms for propagation of evidence in a Bayesian network, studying chordality of the phylogeny graphs of acyclic digraphs is meaningful.
Yet, characterizing the acyclic digraphs whose phylogeny graphs are chordal seems to be not easier than
characterizing the acyclic digraphs whose competition graphs are interval.
In this respect, hoping to provide insights for the further research, we begin with ``$(2,2)$ digraphs" to attack the problem.
A \emph{$(2,2)$ digraph} is an acyclic digraph such that each vertex has both outdegree and indegree at most two.
Hefner~{\it et al.}~\cite{Kim} characterized $(2,2)$ digraphs whose competition graphs are interval.

In this paper, we study the phylogeny graphs of $(2,2)$ digraphs.
Especially, we give a sufficient condition and necessary conditions for $(2,2)$ digraphs having chordal phylogeny graphs.

\section{Preliminaries}

\subsection{Properties of chordal graphs}

We first see some properties of chordal graphs
which will be used in characterizing the $(2,2)$ digraphs
whose phylogeny graphs are chordal.
The following propositions are easy to check.

\begin{Prop} \label{prop:3}
Any induced subgraph of a chordal graph is also a chordal graph.
\end{Prop}

\begin{Prop} \label{prop:2}
A chordal graph containing a cycle of length $n$ has at least $2n-3$ edges.
\end{Prop}

For a graph $G$ and a vertex $v$ of $G$,
a \emph{neighbor} of $v$ in $G$ is a vertex adjacent to $v$ in $G$,
and the set of neighbors of $v$ in $G$ is denoted by $N_G(v)$.
Then the following holds.

\begin{Prop} \label{prop:1}
Let $G$ be a chordal graph.
For any cycle $C$ in $G$ and any edge $xy$ on $C$,
there exists a vertex on $C$ that is a common neighbor of $x$ and $y$ in $G$.
\end{Prop}

\begin{proof}
We show by contradiction.
Suppose that
there exists a cycle $C$ in $G$ which has
an edge $xy \in E(C)$ satisfying $N_G(x) \cap N_G(y) \cap V(C) = \emptyset$.
We take a cycle of the shortest length among such cycles.
Let $C :=v_0v_1\ldots v_{k-1}v_0$ be such a cycle.
If $k=3$, then
$N_G(x) \cap N_G(y) \cap V(C) \neq \emptyset$ for any $xy \in E(C)$. Thus $k \geq 4$.
Since $G$ is chordal, there exists a chord $v_iv_j$ of $C$, where $i<j$.
Let $P_1$ and $P_2$ be the two $(v_i,v_j)$-sections of $C$ and consider the cycles $C_1 := P_1+v_iv_j$ and $C_2 := P_2 + v_iv_j$.
Since both $C_1$ and $C_2$ have lengths shorter than $k$,
$N_G(x) \cap N_G(y) \cap V(C_t) \neq \emptyset$ for any $xy \in E(C_t)$ and $t = 1,2$.
This implies that
$N_G(x) \cap N_G(y) \cap V(C) \neq \emptyset$ for any $xy \in E(C)$,
which is a contradiction.
\end{proof}

We say that a vertex $v$ on a cycle $C$ of length at least $4$ in a chordal graph $G$ is a \emph{vertex opposite to a chord of $C$} if the two vertices immediately following and immediately preceding it, respectively, in the sequence of $C$ are adjacent.
(For example, each of the vertices $v_2$ and $v_5$ in Figure~\ref{fig:eg2} is a vertex opposite to a chord.)

\begin{figure}
\psfrag{a}{$v_1$}
\psfrag{b}{$v_2$}
\psfrag{c}{$v_3$}
\psfrag{d}{$v_4$}
\psfrag{e}{$v_5$}
\centering{
\includegraphics[height=3cm]{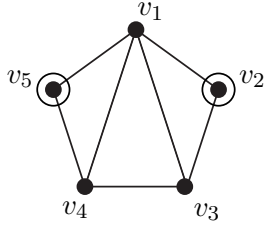}}
\caption{Each of $v_2$ and $v_5$ is a vertex opposite to a chord of the cycle $v_1v_2v_3v_4v_5v_1$.}
\label{fig:eg2}
\end{figure}

\begin{Prop} \label{prop:5}
Each cycle of length at least $4$ in a chordal graph
has at least two nonconsecutive vertices
each of which is opposite to a chord of the cycle.
\end{Prop}

\begin{proof}
Let $G$ be a chordal graph,
$C$ be a cycle of length at least $4$ in $G$,
and $G'$ be the subgraph of $G$ induced by the set of vertices of $C$,
that is, $G'=G[V(C)]$.
Then $G'$ is chordal by Proposition~\ref{prop:3}.
If $G'$ is complete, then the statement is trivially true.
Suppose that $G'$ is not complete.
As Dirac showed that every non-complete chordal graph has at least two nonadjacent simplicial vertices in 1961,
there exist two nonadjacent simplicial vertices $u$ and $v$ in $G'$.
Then each of $u$ and $v$ is obviously opposite to a chord of $C$.
Since $u$ and $v$ are not adjacent in the induced subgraph $G'$ of $G$,
$u$ and $v$ are not consecutive on $C$.
\end{proof}

We call a graph isomorphic to the graph
$G$ defined by $V(G) = \{v_1, \ldots, v_7\}$
and $E(G) = \{ v_iv_j \mid 1 \leq i < j \leq 7, j-i \leq 2 \}$
a \emph{W-configuration} (see Figure~\ref{fig:eg5}).

\begin{figure}
\centering{
\includegraphics[height=3cm]{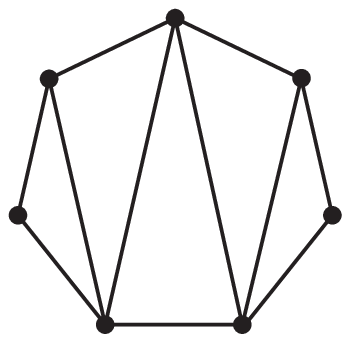}}
\caption{A $W$-configuration}
\label{fig:eg5}
\end{figure}

\begin{Prop} \label{prop:6}
If a chordal graph with the degree of each vertex at most four contains a cycle of length $7$,
then it contains a $W$-configuration as a subgraph.
\end{Prop}

\begin{proof}
Let $G$ be a chordal graph with the degree of each vertex at most four
and let $C: = v_1 \cdots v_7 v_1$ be a cycle of length $7$ of $G$,
and let $H$ be the subgraph of $G$ induced by the vertex set of $C$.
Since $G$ is chordal, so is $H$ by Proposition~\ref{prop:3}.
Since $H$ is a chordal graph containing a cycle of length $7$,
$|E(H)| \geq 11$ by Proposition~\ref{prop:2}.
Then
\[
\sum_{v \in V(H)} \deg_H(v) = 2 |E(H)| \geq 2 \times 11 = 22.
\]
Therefore there exists a vertex on $C$ of degree $4$ in $H$ by the pigeon-hole principle and the hypothesis.
Without loss of generality, we may assume that $v_1$ is a vertex of degree $4$.
By symmetry, it is sufficient to consider the following cases for
the possible pairs of neighbors of $v_1$ other than $v_2$ and $v_7$:
(a) $\{ v_5, v_6 \}$;
(b) $\{ v_4, v_6 \}$;
(c) $\{ v_3, v_6 \}$;
(d) $\{ v_4, v_5 \}$
(see Figure~\ref{fig:eg6}).

\begin{figure}[h]
\psfrag{A}{(a)}
\psfrag{B}{(b)}
\psfrag{C}{(c)}
\psfrag{D}{(d)}
\psfrag{1}{$v_1$}
\psfrag{2}{$v_2$}
\psfrag{3}{$v_3$}
\psfrag{4}{$v_4$}
\psfrag{5}{$v_5$}
\psfrag{6}{$v_6$}
\psfrag{7}{$v_7$}
\centering{
\includegraphics[height=3.5cm]{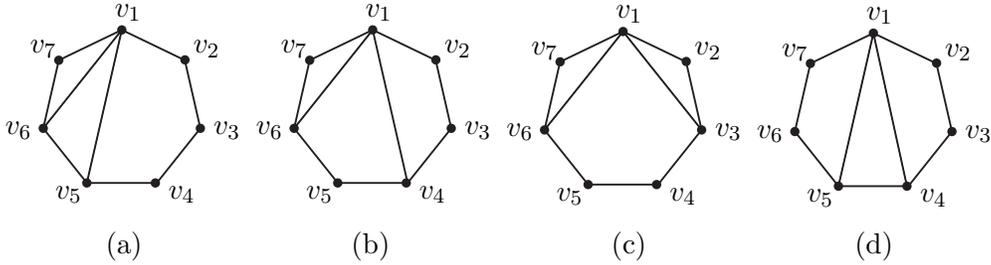}}
\caption{The possible neighbors of $v_1$}
\label{fig:eg6}
\end{figure}
As a matter of fact, the cases (b) and (c) cannot happen.
Suppose to the contrary that (b) happened.
Then each of the $4$-cycles $v_1v_4v_5v_6v_1$ and $v_1v_2v_3v_4v_1$ must contain a chord.
By the way, $v_1$ already has degree $4$, so $v_4v_6$ and $v_4v_2$ must be the chords of $v_1v_4v_5v_6v_1$ and $v_1v_2v_3v_4v_1$, respectively.
Then $\deg_G(v_4)\geq 5$, which is a contradiction.
Thus the case (b) cannot happen.
Suppose that (c) happened.
Then the $5$-cycle $v_1v_3v_4v_5v_6v_1$ must contain a chord since $G$ is chordal.
By the hypothesis that each vertex of $G$ has degree at most $4$, $v_3v_5$, $v_4v_6$, $v_3v_6$ are the only possible chords.
If $v_3v_6$ is a chord, then $v_3v_4v_5v_6v_3$ is a hole, which is impossible.
Thus $v_3$ and $v_6$ are not joined.
Now if $v_3v_5$ is a chord, then $v_1v_3v_5v_6v_1$ is a hole, which is impossible again.
Thus $v_4v_6$ is a chord.
Then $v_1v_3v_4v_6v_1$ is a hole and we reach a contradiction.
Hence the case (c) cannot happen.

Now we consider the case (a).
By applying Proposition~\ref{prop:1} to the edge $v_1v_2$ on the cycle $v_1v_2v_3v_4v_5 v_1$,
we conclude that one of $v_3$, $v_4$, $v_5$ is a vertex that is adjacent to both $v_1$ and $v_2$.
Since each vertex has degree at most $4$ by the hypothesis, it must be $v_5$.
By the same proposition applied to the edge $v_2v_3$ on the cycle $v_2v_3v_4v_5v_2$, both of $v_2$ and $v_3$ are adjacent to $v_4$ or $v_5$.
Since $v_5$ is adjacent to four vertices, the possibility of $v_5$ is eliminated and so $v_4$ is adjacent to $v_2$ and $v_3$.
Thus $G$ contains a W-configuration.

We consider the case (d).
By applying Proposition~\ref{prop:1} to the edge $v_1v_5$ on the cycle $v_1v_5v_6v_7v_1$,
we conclude that $v_6$ or $v_7$ is a vertex that is adjacent to both $v_1$ and $v_5$.
Since $v_1$ is already adjacent to four vertices other than $v_6$, $v_6$ is excluded
and so $v_7$ is adjacent to both $v_1$ and $v_5$.
By applying Proposition~\ref{prop:1} to the cycle $v_1v_2v_3v_4v_1$ and the edge $v_1v_4$,
we conclude that $v_2$ is adjacent to both $v_1$ and $v_4$.
Consequently we obtain a W-configuration.
\end{proof}

\subsection{Induced edges of the phylogeny graphs of $(i,j)$ digraphs}

We call an edge in the phylogeny graph $P(D)$ a \emph{cared edge} in $P(D)$
if the edge belongs to the competition graph $C(D)$ but not to the underlying graph $U(D)$.
For a cared edge $xy$ in $P(D)$,
there is a common out-neighbor $v$ of $x$ and $y$ in $D$ by definition.
The vertex $v$ is called a \emph{vertex taking care of the edge $xy$}
and it is said that \emph{$xy$ is taken care of by $v$} or that \emph{$v$ takes care of $xy$}.
A vertex in $D$ is called an \emph{caring vertex}
if an edge of $P(D)$ is taken care of by the vertex.
For example, the edges $v_1v_7$, $v_7v_8$, and $v_2v_5$ of  $P(D)$ in Figure~\ref{fig:eg31}(c) are cared edges
and the vertices $v_6$, $v_9$, and $v_3$ are vertices taking care of $v_1v_7$, $v_7v_8$, and $v_2v_5$, respectively.

We call an acyclic digraph $D$ an \emph{$(i,j)$ digraph}
if each vertex of $D$ has indegree at most $i$ and outdegree at most $j$.
In the following, we study the structure of the phylogeny graph of a $(i,j)$ digraph.

\begin{Lem} \label{lem:7-0}
Given an $(i,j)$ digraph $D$,
there is no vertex that takes care of
more than $\frac{1}{2}i(i-1)$ edges in the phylogeny graph of $D$.
\end{Lem}

\begin{proof}
Suppose to the contrary that
there exists a vertex $x$ taking care of $t$ different edges for $t \geq \frac{1}{2}i(i-1) +1$.
Then these edges belong to the clique $K$ in the phylogeny graph of $D$ formed by the in-neighbors of $x$ in $D$.
Thus $K$ contains at least $i+1$ vertices. Hence the indegree of $x$ is greater than $i$,
which contradicts the hypothesis that $D$ is an $(i,j)$ digraph.
\end{proof}

\noindent
The following is a consequence of Lemma~\ref{lem:7-0} when $(i,j)=(2,2)$.

\begin{Cor} \label{lem:7}
Given a $(2,2)$ digraph $D$,
there is no vertex that takes care of more than one cared edge in the phylogeny graph of $D$.
\end{Cor}

\begin{Lem} \label{lem:8-0}
Given an $(i,j)$ digraph $D$,
there is no vertex that is incident to more than
$\frac{1}{2}i(i-1)j$ distinct cared edges in the phylogeny graph of $D$.
\end{Lem}

\begin{proof}
Take a vertex $x$ incident to at least one cared edge and let $e_1, \ldots, e_{t}$
be the cared edges in the phylogeny graph $P(D)$ of $D$ incident to $x$,
where $t$ is a positive integer.
Let $w_1, \ldots, w_{s}$ be distinct vertices in $D$
taking care of $e_1, \ldots, e_{t}$, where $s$ is a positive integer.
Then $w_1, \ldots, w_{s}$ are out-neighbors of $x$, so $s \le j$.
By Lemma~\ref{lem:7-0}, each of the vertices $w_1, \ldots, w_{s}$ can take care of at most $\frac{1}{2}i(i-1)$ edges in $P(D)$.
Therefore $t \le \frac{1}{2}i(i-1)s \leq \frac{1}{2}i(i-1)j$ and thus the lemma holds.
\end{proof}

\noindent
The following is a consequence of Lemma~\ref{lem:8-0} when $(i,j)=(2,2)$.

\begin{Cor}\label{lem:8}
Given a $(2,2)$ digraph $D$,
there is no vertex that is incident to three cared edges in the phylogeny graph of $D$.
\end{Cor}

\section{Main Results}
\subsection{The phylogeny graphs of $(2,2)$ digraphs are $K_5$-free}

In this subsection,
we show that there is no $(2,2)$ digraph
whose phylogeny graph contains a complete graph $K_n$ for $n \geq 5$.

Note that we can construct a $(2,2)$ digraph whose phylogeny graph contains
a complete graph $K_4$ as a subgraph
as shown in Example~\ref{eg:1}.

\begin{Exam} \label{eg:1}
Let $D$ be a digraph defined by
$V(D) = \{ v_1, v_2, v_3, v_4, v_5 \}$ and
$A(D) = \{ (v_1, v_2),$ $(v_1, v_4),$ $(v_2, v_3),$ $(v_2, v_5), (v_3, v_4), (v_4, v_5) \}$
(see Figure~\ref{fig:eg17}(a)).
Then $D$ is a $(2,2)$ digraph, and
the subgraph of $P(D)$ induced by $\{ v_1, v_2, v_3, v_4 \}$
is isomorphic to $K_4$
(see Figure~\ref{fig:eg17}(b)).
\end{Exam}

\begin{figure}
\psfrag{a}{(a)}
\psfrag{b}{(b)}
\centering{
\includegraphics[height=3cm]{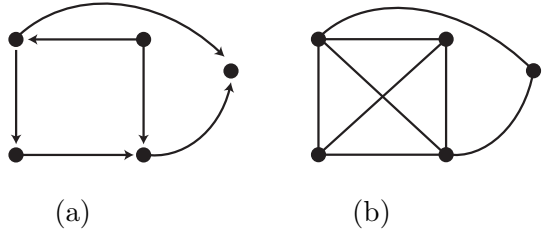}}
\caption{(a) A $(2,2)$ digraph $D$, (b) The phylogeny graph of $D$ contains $K_4$}
\label{fig:eg17}
\end{figure}

\begin{Thm}\label{thm:18}
For any $(2,2)$ digraph $D$,
the phylogeny graph of $D$ is $K_5$-free.
\end{Thm}

\begin{proof}
Suppose to the contrary that
there is a $(2,2)$ digraph $D$ whose phylogeny graph $P(D)$ contains $K_5$ as a subgraph.
Let $v_1,v_2,v_3,v_4,v_5$ be the vertices of $K_5$.
Let $D_1$ be the subdigraph of $D$ induced by $\{v_1,\dots , v_5\}$.
Since $D_1$ is acyclic, there is a vertex of indegree $0$ in $D_1$.
Without loss of generality, we may assume that $v_1$ has indegree $0$.
Now consider the edges $v_1v_2$,$v_1v_3$,$v_1v_4$,$v_1v_5$ in $P(D)$.
At most two of them  are cared edges by Corollary~\ref{lem:8}.
Therefore, at least two of $v_1v_2$,$v_1v_3$,$v_1v_4$,$v_1v_5$ belong to $U(D)$.
By the way, they must be arcs outgoing from $v_1$ since $v_1$ has indegree $0$ in $D_1$.
Since $v_1$ has outdegree at most two in $D$, $v_1$ has outdegree exactly two in $D$.
Therfore exactly two of $v_1v_2$,$v_1v_3$,$v_1v_4$,$v_1v_5$ belong to $U(D)$ and, consequently,
the remaining two edges are cared edges in $P(D)$.

Without loss of generality, we may assume that $v_1v_2$ and $v_1v_3$ are cared edges.
Then $v_4$ and $v_5$ are the out-neighbors of $v_1$ in $D$.
Since a vertex taking care of $v_1v_2$ (resp.\ $v_1v_3$) is an out-neighbor of
$v_1$, $v_1v_2$ (resp.\ $v_1v_3$) is taken care of by $v_4$ or $v_5$.
Without loss of generality, we may assume that $v_4$ is a common out-neighbor of $v_1$ and $v_2$ in $D$.
Then, by Corollary~\ref{lem:7}, $v_5$ is a common out-neighbor of $v_1$ and $v_3$.
However, since both $v_4$ and $v_5$ have indegree two, the edge $v_4v_5$ cannot belong to $U(D)$ and so is a cared edge in $P(D)$.
Since $(v_1,v_4)$, $(v_2,v_4)$, $(v_1,v_5)$, and $(v_3,v_5)$ are arcs of $D$,
none of $v_1$, $v_2$, $v_3$ is a common out-neighbor of $v_4$ and $v_5$ by acyclicity of $D$.
Therefore, the edge $v_4v_5$ is taken care of by a vertex $w_1$ distinct from $v_1,v_2,v_3$ (see Figure~\ref{fig:eg12}).

\begin{figure}
\psfrag{1}{$v_1$}
\psfrag{2}{$v_2$}
\psfrag{3}{$v_3$}
\psfrag{4}{$v_4$}
\psfrag{5}{$v_5$}
\psfrag{w}{$w_1$}
\centering{
\includegraphics[height=3.5cm]{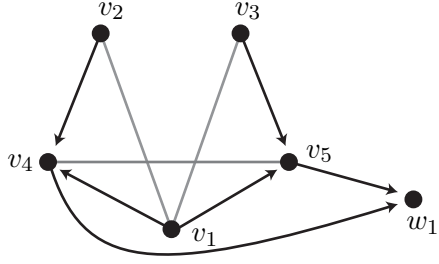}}
\caption{The vertex $w_1$ taking care of $v_4v_5$. The grey edges are cared edges. }
\label{fig:eg12}
\end{figure}

We show that at least one of $v_3v_4,v_2v_5$ is a cared edge.
Suppose to the contrary that
both $v_3v_4$ and $v_2v_5$ are not cared edges.
Then $v_3v_4$ and $v_2v_5$ are inherited from arcs $(v_4,v_3)$ and $(v_5,v_2)$, respectively,
since the indegrees of $v_4$ and $v_5$ are at most two.
Therefore $v_2 \rightarrow v_4 \rightarrow v_3 \rightarrow v_5 \rightarrow v_2$ is a directed cycle,
which contradicts the hypothesis that $D$ is acyclic.
Thus at least one of $v_3v_4,v_2v_5$ is a cared edge.
Without loss of generality, we may assume that $v_3v_4$ is a cared edge.
Since none of $v_2$ and $v_5$ can be an out-neighbor of $v_4$, neither $v_2$ nor $v_5$ takes care of $v_3v_4$.
Since the indegree of $v_1$ is $0$, $v_1$ does not take care of $v_3v_4$ either.
Since $w_1$ has in-neighbors $v_4$ and $v_5$, $w_1$ cannot take care of the edge $v_3v_4$ due to the degree condition imposed on $D$.
Therefore the edge $v_3v_4$ is taken care of by a vertex $w_2$ distinct from $v_1,v_2,v_5, w_1$.
Recall that $v_3v_4$ and $v_3v_1$ are cared edges.
Since $v_3$ cannot be incident to three cared edges by Corollary~\ref{lem:8}, the edge $v_2v_3$ belongs to $U(D)$.
Then $v_2$ is an in-neighbor of $v_3$ since the outdegree of $v_3$ is at most two.

Finally we take a look at the edge $v_2v_5$.
Since the indegree of $v_5$ is two that is achieved by $v_1$ and $v_3$, $v_2$ cannot be an in-neighbor of $v_5$.
If $v_2$ is an out-neighbor of $v_5$, then it results in the directed cycle $v_2\to v_3\to v_5\to v_2$, which is a contradiction.
Therefore $v_2v_5$ is a cared edge. Since $v_3$ and $v_4$ are the only out-neighbors of $v_2$, $v_3$ or $v_4$ takes care of $v_2v_5$.
Since $v_5$ is an out-neighbor of $v_3$, $v_3$ cannot take care of the edge $v_2v_5$.
Since the edge $v_4v_5$ is a cared edge, $v_4$ cannot take care of the edge $v_2v_5$ either.
Hence we have reached a contradiction.
\end{proof}

\subsection{A necessary condition for the phylogeny graph of a $(2,2)$ digraph being chordal}

Properties of $(2,2)$ digraphs make us speculate that sufficiently long hole in the underlying graph of a $(2,2)$ digraph $D$
might give rise to a hole in the phylogeny graph of $D$ as chords cannot be produced enough to fill in it.
This motivates us to find the length of a shortest hole among holes in the underlying graph of a $(2,2)$ digraph whose phylogeny graph is chordal.

\begin{Lem}\label{lem:subgraph}
Let $D$ be a $(2,2)$ digraph.
If the underlying graph of $D$ contains a hole $H$ of length at least $7$,
then the subgraph of the phylogeny graph of $D$ induced by $V(H)$ is not chordal.
\end{Lem}

\begin{proof}
Let $H = v_1v_2\cdots v_nv_1$ ($n \ge 7$) be a hole of length at least $7$ in $U(D)$
and let $G_1$ be the subgraph of $P(D)$ induced by $V(H) = \{v_1, \ldots ,v_n\}$.
Note that no edges on $H$ are cared edges while the edges of $G_1$ not on $H$ are cared edges in $P(D)$.
Suppose to the contrary that $G_1$ is chordal.
If there exists a vertex $v$ with $\deg_{G_1}(v)\geq 5$, then $v$ is incident to at least three cared edges in $G_1$ and so in $P(D)$,
which contradicts Corollary~\ref{lem:8}.
Therefore $d_{G_1}(v)\leq 4$ for every vertex $v$ in $G_1$.

Since $G_1$ is a hamiltonian chordal graph with at least seven vertices, there exists a vertex on $G_1$ opposite to a chord by Proposition~\ref{prop:5}.
Without loss of generality, we may assume that $v_2$ is such a vertex.
Let $G_2$ be the graph obtained by deleting $v_2$ from $G_1$.
Since $v_2$ is a vertex opposite to a chord, $G_2$ is a hamiltonian chordal graph with $n-1$ vertices with a hamiltonian cycle $v_1v_3v_4\cdots v_nv_1$.
Since $n-1 \geq 6$, we may apply Proposition~\ref{prop:5} again to have a vertex on $G_2$ which is opposite to a chord.
We delete one of such vertices from $G_2$ to obtain a hamiltonian chordal graph with $n-2$ vertices.
We continue this process until we obtain a hamiltonian chordal graph $G^*$ with $7$ vertices.
Let $v_{n_1}v_{n_2}\cdots v_{n_7}v_{n_1}$ be a hamiltonian cycle of $G^*$ with $n_1< n_2< \cdots < n_7$,
which exists by the definition of $G^*$.
By Proposition~\ref{prop:6}, $G^*$ is a $W$-configuration.
Without loss of generality, we may assume that it is labeled as in Figure~\ref{fig:eg7}.
\begin{figure}
\psfrag{a}{(a)}
\psfrag{b}{(b)}
\psfrag{1}{$v_{n_1}$}
\psfrag{2}{$v_{n_2}$}
\psfrag{3}{$v_{n_3}$}
\psfrag{4}{$v_{n_4}$}
\psfrag{5}{$v_{n_5}$}
\psfrag{6}{$v_{n_6}$}
\psfrag{7}{$v_{n_7}$}
\centering{
\includegraphics[height=4cm]{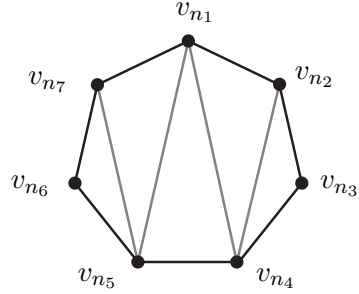}}
\caption{A $W$-configuration of $G^*$. The grey edges are cared edges.}
\label{fig:eg7}
\end{figure}

If the end vertices of an edge in $G^*$ are not on the hamiltonian cycle $v_{n_1}v_{n_2}\cdots v_{n_7}v_{n_1}$,
then the index difference of them is neither $1$ nor $n-1$ since $n_1< n_2< \cdots < n_7$.
Noting that an edge $v_mv_k$ $(1\leq m,k \leq n)$ is on $H$ if and only if $|m-k|=1$ or $|m-k|=n-1$,
we may conclude that $v_{n_4}v_{n_1}$,  $v_{n_4}v_{n_2}$, $v_{n_5}v_{n_1}$, and $v_{n_5}v_{n_7}$ are cared edges.
Let $w_1$, $w_2$, $w_3$, $w_4$ be caring vertices of the edges $v_{n_4}v_{n_1}$, $v_{n_4}v_{n_2}$, $v_{n_5}v_{n_1}$, $v_{n_5}v_{n_7}$, respectively.
Then $w_1$, $w_2$, $w_3$, $w_4$ are all distinct by Corollary~\ref{lem:7},
and $w_1$, $w_2$ (resp.\ $w_3$, $w_4$) are out-neighbors of $v_{n_4}$ (resp.\ $v_{n_5}$).
Since $D$ is a $(2,2)$ digraph, there cannot exist an arc between $v_{n_4}$ and $v_{n_5}$.
Therefore $v_{n_4}v_{n_5}$ should be a cared edge.
Then $v_{n_4}$ is incident to three cared edges, which contradicts Lemma~\ref{lem:8}.
Hence $G_1$ is not chordal.
\end{proof}

\noindent By Lemma~\ref{lem:subgraph} and Proposition~\ref{prop:3}, the following theorem holds.

\begin{Thm}\label{thm:9}
Let $D$ be a $(2,2)$ digraph.
If the underlying graph of $D$ contains a hole of length at least $7$,
then the phylogeny graph of $D$ is not a chordal graph.
\end{Thm}

\subsection{Forbidden subdigraphs for the class of $(2,2)$ digraphs whose phylogeny graphs are chordal}

Let $\mathcal{D}$ be a class of digraphs.
A digraph $D_0$ is called
a \emph{forbidden subdigraph} for $\mathcal{D}$ if
$D \not\in \mathcal{D}$ holds
for any digraph $D$ containing $D_0$ as an induced subdigraph.
Note that, by Lemma~\ref{lem:subgraph}, the orientations of cycles of length at least $7$ are forbidden subdigraphs for the class $\mathcal{D}^*$ of $(2,2)$ digraphs
whose phylogeny graphs are chordal. In this subsection,
we determine the non-isomorphic orientations of cycles of length $4$ or $5$ or $6$ that are forbidden subdigraphs for $\mathcal{D}^*$.

\begin{figure}
\psfrag{a}{(a)}
\psfrag{b}{(b)}
\psfrag{c}{(c)}
\psfrag{d}{(d)}
\psfrag{e}{(e)}
\centering{
\includegraphics[height=7cm]{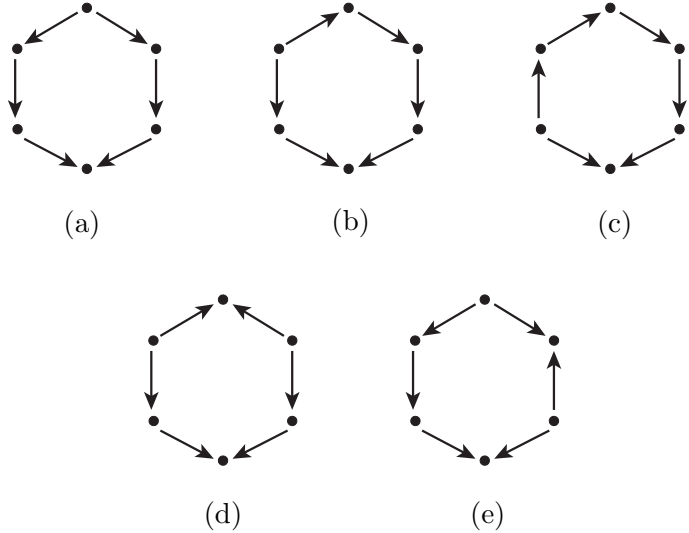}}
\caption{The forbidden subdigraphs among orientations of cycles of length at most six
for the class of $(2,2)$ digraphs whose phylogeny graphs are chordal.}
\label{fig:forbidden}
\end{figure}

\begin{Thm}
Let $\mathcal{D}^*$ be the class of $(2,2)$ digraphs
whose phylogeny graphs are chordal.
Then the digraphs given in Figure~\ref{fig:forbidden} are the forbidden subdigraphs
among orientations of cycles of length at most six for $\mathcal{D}^*$.
\end{Thm}

\begin{proof}
Suppose that a $(2,2)$ digraph $D$ contains an  orientation $C$ of a cycle with length six given in Figure~\ref{fig:forbidden} as an induced subdigraph.
(We provided the chordal phylogeny graph of a $(2,2)$ digraph
containing each of orientations of cycles of length $4$ or $5$ or $6$ in Figures~\ref{fig:4cycles}, \ref{fig:5cycles}, \ref{fig:6cycles}
other than the ones given in Figure~\ref{fig:forbidden}.)
Let $S$ be the subgraph of the phylogeny graph $P(D)$ of $D$ induced by $V(C)$.
To reach a contradiction, suppose that $P(D)$ is chordal.
Then, in case of (a), (b), (c), the subgraph of $P(D)$ induced by the vertex set of a cycle $H_1$ of length five is contained in $S$,
so it has at least two adjacent chords which are taken care of by two caring vertices.
Since each vertex in $D$ has indegree at most two, the two caring vertices must be distinct.
In case of (d), (e), $S$ contains the subgraph of $P(D)$ induced by the vertex set of a cycle $H_2$ of length four, so it has a chord which are taken care of by a caring vertex.
Since $C$ is an induced subdigraph of $D$, neither the vertices taking care of chords of $H_1$ nor the vertices taking care of chords of $H_2$ can be on $C$.
Therefore, in case of (a), (b), (c),
the vertex common to the two adjacent chords of $H_1$ must have two out-neighbors not on $C$
and in case of (d), (e), there are two nonadjacent vertices on $H_2$ each of which has an out-neighbor not on $C$.
However, by the structure of $C$, each vertex on $H_1$ has an out-neighbor on $C$
and especially in case of (d), (e), there are two adjacent vertices on $H_2$ each of which has two out-neighbors on $C$ in $D$.
Hence, in either case, we obtain a vertex of outdegree at least three, which is a contradiction.
\end{proof}

\begin{figure}[!!!h]
\psfrag{a}{(a)}
\psfrag{b}{(b)}
\psfrag{c}{(c)}
\centering{
\includegraphics[scale=0.5]{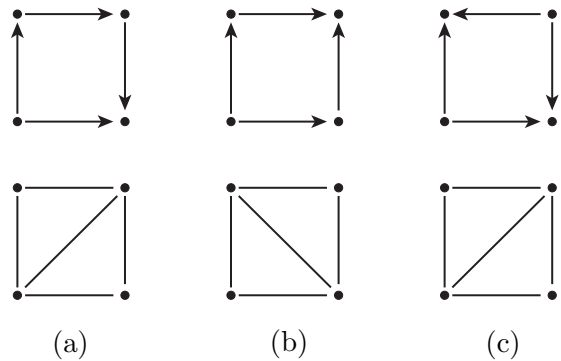}}
\caption{The non-isomorphic orientations of cycles of length $4$ and their corresponding phylogeny graphs which are chordal.}
\label{fig:4cycles}
\end{figure}

\begin{figure}[!!!h]
\psfrag{a}{(a)}
\psfrag{b}{(b)}
\psfrag{c}{(c)}
\psfrag{d}{\small $v_1$}
\psfrag{e}{\small $v_2$}
\psfrag{f}{\small $v_3$}
\psfrag{g}{\small $v_4$}
\psfrag{h}{\small $v_5$}
\centering{
\includegraphics[scale=0.5]{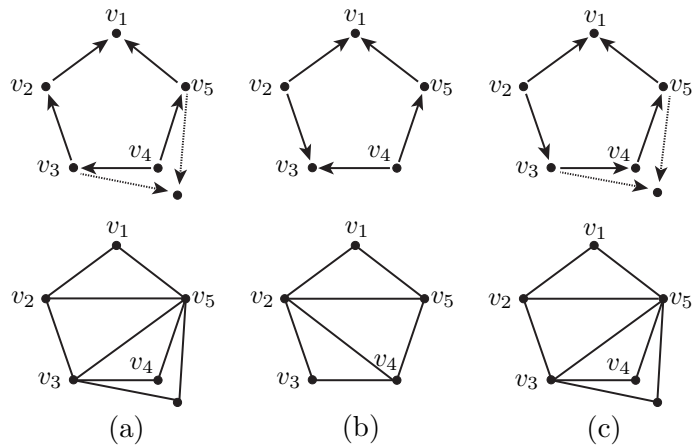}}
\caption{The non-isomorphic acyclic orientations of the 5-cycle $v_1v_2v_3v_4v_5v_1$ and chordal phylogeny graphs of acyclic digraphs including them as induced subdigraphs.}
\label{fig:5cycles}
\end{figure}

\begin{figure}[!!!h]
\psfrag{a}{(a)}
\psfrag{b}{(b)}
\psfrag{c}{(c)}
\psfrag{d}{\small $v_1$}
\psfrag{e}{\small $v_2$}
\psfrag{f}{\small $v_3$}
\psfrag{g}{\small $v_4$}
\psfrag{h}{\small$v_5$}
\psfrag{k}{\small $v_6$}
\centering{
\includegraphics[scale=0.5]{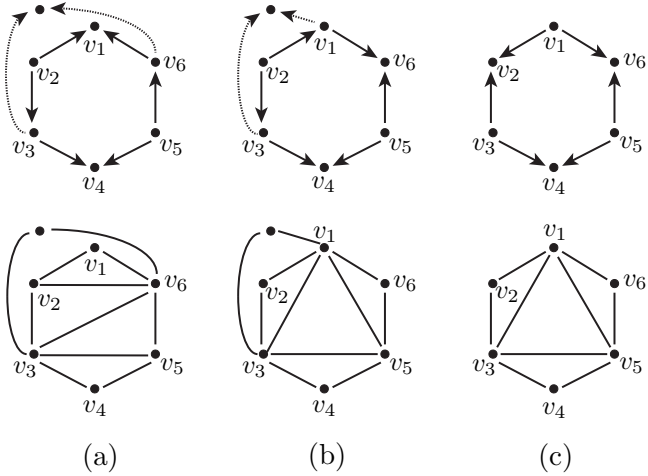}}
\caption{The non-isomorphic acyclic orientations of the $6$-cycle $v_1v_2v_3v_4v_5v_6v_1$ satisfying the property that the phylogeny graph of an acyclic digraph
including one of them as an induced subdigraph is chordal and their corresponding phylogeny graphs which are chordal.}
\label{fig:6cycles}
\end{figure}

\subsection{Holes in the underlying graph and the phylogeny graph of a $(2,2)$ digraph}

As the edges on holes in the underlying graph of a digraph are inherited to its phylogeny graph,
one may expect that the phylogeny graph cannot have a hole longer than the ones in the underlying graph.
Contrary to this expectation, each hole in the underlying graph of $D$ in Figure~\ref{fig:eg18}(a) has length $4$
while the hole $H=v_1v_2v_3v_6v_7v_4v_1$ in its phylogeny graph $P(D)$ in Figure~\ref{fig:eg18}(b) has length $6$.
However, the phylogeny graph of a $(2,2)$ digraph lives up to the expectation as long as its underlying graph is chordal.
Before we prove it, we derive the following statements to be used in the proof.

\begin{figure}[h]
\psfrag{1}{$v_1$}\psfrag{2}{$v_2$}
\psfrag{3}{$v_3$}\psfrag{4}{$v_4$}
\psfrag{5}{$v_5$}\psfrag{6}{$v_6$}
\psfrag{7}{$v_7$}\psfrag{8}{$v_8$}
\psfrag{w}{$w_1$}\psfrag{x}{$w_2$}
\psfrag{a}{(a)}\psfrag{b}{(b)}
\centering{
\includegraphics[height=5cm]{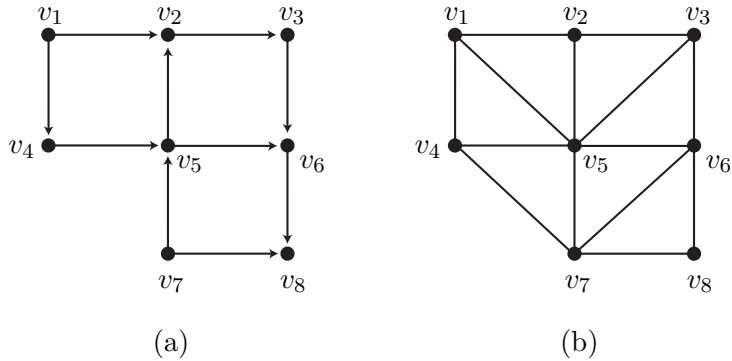}}
\caption{(a) A $(2,2)$ digraph $D$, (b) The phylogeny graph of $D$}
\label{fig:eg18}
\end{figure}

\begin{Prop}\label{prop:10}
Suppose that the phylogeny graph of a $(2,2)$ digraph $D$ contains a hole $H$.
If $v$ is a vertex taking care of an edge on $H$,
then $v$ is not on $H$.
\end{Prop}

\begin{proof}
Suppose to the contrary that there exists a vertex $v$ on $H$ taking care of an edge $xy$ on $H$.
Then the edge $xv$ or the edge $yv$ is a chord of $H$, which is a contradiction.
\end{proof}

Given a $(2,2)$ digraph $D$,
suppose that the phylogeny graph $P(D)$ has a hole $H$ of length $n$ for $n\geq 4$
and $e_1,e_2, \ldots, e_m$ are the cared edges of $H$.
Let $w_1,w_2,\dots, w_m$ be vertices taking care of $e_1,e_2, \ldots, e_m$, respectively, and $W = \{w_1,w_2,\ldots, w_m\}$.
We call $W$ a \emph{set extending $H$}.
Then $W \subseteq V(D)-V(H)$ by Proposition~\ref{prop:10}.
We may obtain a cycle in $U(D)$ from $H$ by replacing each edge $e_i$
with a path of length two from one end of $e_i$ to the other end of $e_i$ with the interior vertex $w_i$.
We call such a cycle the \emph{cycle obtained from $H$ by $W$}.
Let $L$ be the subgraph of $U(D)$ induced by $V(H) \cup W$.
We call $L$ the \emph{subgraph of $U(D)$ obtained from $H$ by $W$}.
By definition, the cycle obtained from $H$ by $W$ is a hamiltonian cycle of the subgraph obtained from $H$ by $W$.
For example the graph in Figure~\ref{fig:eg8}(b)
is the subgraph obtained from $H$ by $\{w_1,w_2\}$ in Figure~\ref{fig:eg8}(a).

\begin{figure}
\psfrag{1}{$e_2$}\psfrag{2}{$e_1$}
\psfrag{3}{$v_1$}\psfrag{4}{$v_2$}
\psfrag{5}{$w_2$}\psfrag{6}{$w_1$}
\psfrag{H}{(a)}\psfrag{G}{(b)}
\psfrag{a}{$v_1$}\psfrag{b}{$v_2$}
\psfrag{c}{$v_3$}\psfrag{d}{$v_4$}
\psfrag{e}{$v_5$}\psfrag{f}{$v_6$}
\centering{
\includegraphics[height=4.5cm]{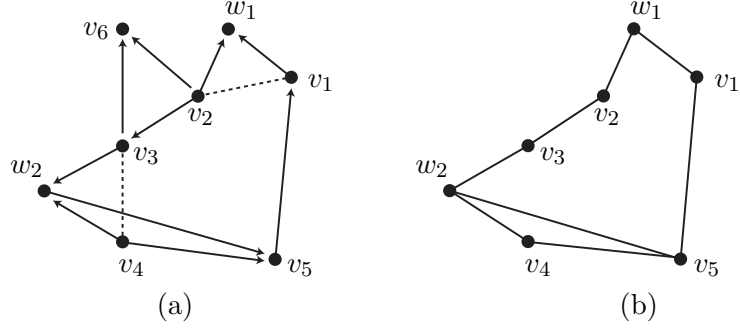}}
\caption{(a) A hole $H=v_1v_2v_3v_4v_5v_1$ in the phylogeny graph of a $(2,2)$ digraph $D$.
(b) The subgraph obtained from $H$ by $\{w_1,w_2\}$.}
\label{fig:eg8}
\end{figure}

\begin{Lem}\label{lem:11}
Suppose that the phylogeny graph of a $(2,2)$ digraph $D$ contains a hole $H$.
If $L$ is the subgraph of the underlying graph $U(D)$ of $D$ obtained from $H$ by a set $W$ extending $H$,
then there is no edge joining two vertices belonging to $W$ in $U(D)$.
\end{Lem}

\begin{proof}
Suppose to the contrary that there exist vertices $w_1,w_2 \in W$ that are adjacent in $U(D)$.
Without loss of generality, we may assume that $(w_1,w_2) \in A(D)$.
Since $L$ is obtained from $H$ by $W$, $w_2$ is a vertex taking care of an edge on $H$.
Thus $w_2$ has two in-neighbors in $D$ which belong to $V(H)$.
Since $(w_1,w_2) \in A(D)$,
$w_2$ has indegree at least three, which is a contradiction.
\end{proof}

\begin{Lem}\label{lem:12}
Let $H$ be a hole in the phylogeny graph $P(D)$ of a $(2,2)$ digraph $D$,
and $L$ be the subgraph of the underlying graph $U(D)$ of $D$ obtained from $H$ by a set $W$ extending $H$.
If $L$ is chordal and $xy \in E(H)$ is an edge in $P(D)$ taken care of by $w \in W$,
then there exists a vertex $z$ on $H$
such that $z$ is adjacent to both $x$ and $w$ in $L$.
\end{Lem}

\begin{proof}
Let $C$ be the cycle obtained from $H$ by $W$.
Then $C$ is a hamiltonian cycle of $L$.
Since $L$ is obtained from $H$ by $W$ containing $w$,
the edge $xw$ is on $C$.
Since $L$ is chordal, there exists a vertex $z \in V(C)-\{x,w\}$
that is adjacent to both $x$ and $w$ in $L$ by Proposition~\ref{prop:1}.
Since $L$ is a subgraph of $U(D)$, $w$ and $z$ are adjacent in $U(D)$.
By Lemma~\ref{lem:11}, $z \notin W$ and so $z$ belongs to $H$, which completes the proof.
\end{proof}

\begin{Thm} \label{Thm:13}
Let $H$ be a hole of the phylogeny graph $P(D)$ of a $(2,2)$ digraph $D$.
Then there is a hole $\phi(H)$ in the underlying graph $U(D)$ of $D$ such that
\begin{itemize}
\item
$\phi(H)$ equals $H$ if $H$ is a hole in $U(D)$;
\item
$\phi(H)$ is a hole in $U(D)$ only containing vertices in the subgraph obtained from $H$ by a set extending $H$ otherwise.
\end{itemize}
Moreover, if the holes of $P(D)$ are mutually vertex-disjoint and no hole in $U(D)$ has length $4$ or $6$,
then there exists an injective map from the set of holes in $P(D)$ to the set of holes in $U(D)$.
\end{Thm}

\begin{proof}
Let $H = v_1v_2\cdots v_nv_1$ be a hole in $P(D)$.
If no edge of $H$ is taken care of, then $H$ is a hole in $U(D)$ and we let $\phi(H)=H$.

Suppose that at least one edge of $H$ is taken care of.
Let $e_1, \ldots, e_m$ be the cared edges of $H$ and let $w_i$ be a vertex taking care of $e_i$ for each $i=1,\ldots,m$.
Let $L$ be the subgraph of $U(D)$ obtained from $H$ by $\{w_1,\ldots,w_m\}$.

To reach a contradiction, suppose that $L$ is chordal.
Without loss of generality, we may assume that $v_1$ and $v_2$ are the end vertices of $e_1$.
By Lemma~\ref{lem:12}, $v_1$ and $w_1$ have a common neighbor in $L$, say $z$, on $H$.
Since $v_1v_2$ is a cared edge, $z \neq v_2$.
Since $H$ is a hole in $P(D)$, the edge $v_1z$ cannot be a chord of $H$ and so $z=v_{n}$.
Therefore $v_nw_1$ and $v_nv_1$ are edges in $L$.
Since $L$ is a subgraph of $U(D)$, $v_nw_1$ and $v_nv_1$ are edges in $U(D)$.
Since $w_1$ has $v_1$ and $v_2$ as in-neighbors and $D$ is a $(2,2)$ digraph,
$v_n$ must be an out-neighbor of $w_1$ in $D$.
Since $D$ is acyclic, $v_n$ is an out-neighbor of $v_1$.
Similarly, $v_3$ is a common out-neighbor of $v_2, w_1$ in $D$ (see Figure~\ref{fig:eg11}).

\begin{figure}
\psfrag{1}{$v_1$}\psfrag{2}{$v_2$}
\psfrag{3}{$v_3$}\psfrag{n}{$v_n$}
\psfrag{w}{$w_1$}
\centering{
\includegraphics[height=3.25cm]{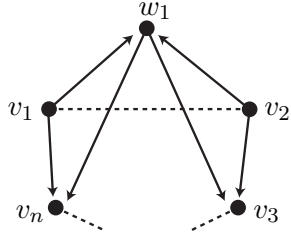}}
\caption{The arcs of $D$ corresponding to $v_1v_n,v_1w_1,v_2v_3,v_2w_1,w_1v_n,w_1v_3$ in $U(D)$}
\label{fig:eg11}
\end{figure}

Now we consider the graph $L^*$ obtained by deleting $v_1,v_2$ from $L$.
We note that $H^* := w_1v_3v_4 \cdots v_n w_1$ is a hole in $P(D)$ and that $L^*$ is the subgraph of $U(D)$ obtained from $H^*$ by $\{w_2, w_3, \ldots, w_n\}$.
By applying Proposition~\ref{prop:3} to $L$, we can conclude that the subgraph $L^*$  is chordal.
For an edge $w_1v_3$ on $L^*$, there exists a vertex $z^*\in V(H^*)$ that is adjacent to both $w_1$ and $v_3$ in $U(D)$ by Lemma~\ref{lem:12}.
Since $z^*w_1$ and $z^*v_3$ are edges of $L^*$ and $L^*$ is a subgraph of $U(D)$, they are edges in $U(D)$.
Then $z^*=v_n$ since $z^*\in N_{U(D)}(w_1)=\{v_1,v_2,v_3,v_n\}$ and $z^*\notin \{v_1,v_2,v_3\}$.
Therefore the edge $z^*v_3$ in $U(D)$ is now $v_nv_3$.
Thus either $(v_3,v_n)$ or $(v_n,v_3)$ is an arc in $D$.
However $v_n\notin \{v_2,w_1\}=N^-_{D}(v_3)$ and $v_3\notin \{v_1,w_1\}=N^-_{D}(v_n)$, which is a contradiction.
Thus $L$ contains a hole, that is, there exists a hole all of whose vertices are in $L$.
We take one of such holes as $\phi(H)$.
Then $\phi$ defines a map from the set of holes in $P(D)$ to the set of holes in $U(D)$.

To show the second part of the theorem, we assume that the holes of $P(D)$ are mutually vertex-disjoint and no hole in $U(D)$ has length $4$ or $6$.
We take $\phi^*$ whose image has the largest size among the maps that can be obtained by the way described in the previous argument.
Then we take two distinct holes $H_1$ and $H_2$ in $P(D)$.
By the hypothesis, $H_1$ and $H_2$ are vertex-disjoint.
Let $L_1$ and $L_2$ be the subgraphs of $U(D)$ obtained from $H_1$ and $H_2$ by sets $W_1$ and $W_2$ extending $H_1$ and $H_2$, respectively.
By the above argument, $\phi^*(H_1)$ (resp.\ $\phi^*(H_2)$) is a hole whose vertices are on $L_1$ (resp.\ $L_2$).
Suppose $\phi^*(H_1)=\phi^*(H_2)=:H^*$.  If $H^*$ contains a vertex neither on $H_1$ nor on $H_2$,
then it is a vertex taking care of an edge on $H_1$ and an edge on $H_2$ at the same time, which contradicts the hypothesis that $D$ is a $(2,2)$ digraph.
Thus $H^*$ consists of vertices on $H_1$ or $H_2$.
Suppose that $H^*$ contains two consecutive vertices both of which are on $H_1$ (resp. $H_2$).
Then they are adjacent by an arc $a$ in $D$.
Since they belong to $H_1$ (resp. $H_2$), they are vertices taking care of edges on $H_2$ (resp. $H_1$) since $H_1$ and $H_2$ are vertex-disjoint.
Therefore each of them has two in-neighbors on $H_2$.
However, due to $a$, one of them must have indegree at least three in $D$ and we reach a contradiction.
Thus the vertices on $H^*$ belong alternatively to $H_1$ and $H_2$ and so $H^*$ is a hole of an even length.
By the hypothesis, $H^*$ has length at least $8$.
Then, by applying Lemma~\ref{thm:9} to the subgraph of $P(D)$ induced by $V(H^*)$, there exists a hole consisting of vertices of $H^*$ in $P(D)$.
Since each vertex on $H^*$ belongs to $H_1$ or $H_2$, by the hypothesis that the holes in $P(D)$ are vertex-disjoint,
this hole is either $H_1$ or $H_2$.
Without loss of generality, we may assume that it is $H_1$.
Then $H^*$ contains all the vertices of $H_1$ and so each vertex on $H_1$ takes care of an edge on $H_2$.
Moreover, since $|V(H_2) \cap V(H^*)|$ is the same as the number of edges on $H_1$, each edge on $H_1$ is taken care of by a vertex in $V(H_2) \cap V(H^*)$.

Let $C$ be the cycle obtained from $H_2$ by $W_2$.
Then
\[V(H_1) \subseteq V(H^*) \subseteq V(C) = V(L_2) = V(H_2) \cup W_2.\]
Now, since $H_1$ and $H_2$ are vertex-disjoint, $V(H_1) \subseteq W_2$ and so each vertex on $H_1$ takes care of an edge on $H_2$.

Take a vertex $u$ on $H_1$. Then $u$ is adjacent to two vertices, say $z_1$ and $z_2$, on $H_1$.
As we claimed that each vertex both on $H^*$ and $H_2$ takes care of each edge of $H_1$,
there are out-neighbors $v$ and $w$ of $u$ such that $v$ and $w$ are on $V(H^*) \cap V(H_2)$ and $(z_1, v)$, $(u,v)$, $(z_2, w)$, $(u,w)$ are arcs in $D$.
Since $v$ and $w$ are caring vertices, they are not adjacent in $U(D)$ by Lemma~\ref{lem:11}.
Since $u$ belongs to $H^*$, $u$ is a  vertex taking care of an edge $xy$ on $H_2$. Then $x$ and $y$ are in-neighbors of $u$ in $D$.
We take the $(v,w)$-section $P$ of $C$ that does not contain $u$.
Then $x$ and $y$, which are consecutive on $H_2$, do not belong to $P$ since $xuy$ is a section of $C$ by the way in which $C$ is obtained.
Thus, by the degree restriction on $D$, $u$ is not adjacent to any vertex on $P$ other than $v$ and $w$ in $U(D)$.
We take a shortest $(v,w)$-path $P^*$ in the subgraph of $U(D)$ induced by the vertex set of $P$.
Since $v$ and $w$ are not adjacent in $U(D)$, $uP^*u$ is a hole in $U(D)$.
Suppose $uP^*u=H^*$.
Then, since it is on $H^*$, $z_1$ is on $P^*$.
If $x$ and $v$ are adjacent in $U(D)$, then, since $v$ has already two in-neighbors $z_1$ and $u$, the edge $xv$ in $U(D)$ has orientation $(v,x)$ to form a directed cycle $u\to v \to x \to u$, which is impossible.
Thus $x$ and $v$ are not adjacent in $U(D)$.
Hence, for the $(u,v)$-section $Q$ of $C$ containing $x$,
$Qu$ contains a hole $H^{**}$ in $U(D)$ containing $u$ and $x$ since $u$ is not adjacent to any vertex other than $x$ and $v$ on $Q$.
Since $x$ is not on $P$, it is not on $uP^*u$.
Then, since $uP^*u=H^*$, $x$ does not belong to $H^*$ and therefore  $H^{**}$ is distinct from $H^*$.
Therefore we can conclude that $uP^*u$ or $H^{**}$ is a hole different from $H^*$ containing $u$ in $U(D)$.
We change $\phi^*(H_2)$ into $uP^*u$ if $H^* \neq uP^*u$ and into $H^{**}$ otherwise.
By the degree restriction on $D$, $u$ belongs to only $L_1$ and $L_2$.
Thus the new $\phi^*(H_2)$ does not equal any of $\phi^*$-values of other holes in $P(D)$
and we have obtained a map from the set of holes in $P(D)$ to a set of holes in $U(D)$
with image larger than $\phi^*$, which contradicts the choice of $\phi^*$.
\end{proof}

\begin{Rem}
The ``Moreover" part of Theorem~\ref{Thm:13} does not hold in general.
For the digraph $D$ given in Figure~\ref{fig:eg8},
the holes in $P(D)$ are $v_1v_2v_3v_4v_5v_1$ and $v_1v_2v_3w_2v_5v_1$ while the hole in $U(D)$ is $v_1w_1v_2v_3w_2v_5v_1$.
\end{Rem}

\begin{Cor}
Let $D$ be a $(2,2)$ digraph.
Suppose that the holes of $P(D)$ are mutually vertex-disjoint and no holes in $U(D)$ has length $4$ or $6$.
Then the number of holes in $U(D)$ is greater than or equal to that of holes in $P(D)$.
\end{Cor}

\begin{Cor}\label{cor:chordal}
Let $D$ be a $(2,2)$ digraph.
If $U(D)$ is chordal, then $P(D)$ is also chordal.
\end{Cor}

\section{Concluding remarks}

In this paper, we obtained the complete list of orientations of cycles that are forbidden subdigraphs for the class of $(2,2)$ digraphs whose phylogeny graphs are chordal.
Furthermore, we showed that
if the holes of the phylogeny graph $P(D)$ of a $(2,2)$ digraph $D$
are mutually vertex-disjoint and no holes in the underlying graph $U(D)$ of $D$ has length $4$ or $6$,
then the number of holes in $U(D)$ is greater than or equal to
that of holes in $P(D)$,
which implies the following: If the underlying graph of a $(2,2)$ digraph $D$ is chordal,
then the phylogeny graph of $D$ is also chordal.
It would be interesting to give a good necessary and sufficient condition for $(2,2)$ digraphs having chordal phylogeny graphs.

\section*{Acknowledgments}

The first author's research was supported by Global Ph.D Fellowship Program
through the National Research Foundation of Korea (NRF)　
funded by the Ministry of Education (No.~NRF-2015H1A2A1033541).
The second author's research was supported by
the National Research Foundation of Korea(NRF) funded by the Korea government(MEST) (No.\ NRF-2015R1A2A2A01006885) and by the Korea government(MSIP)
(No.\ KRF-2016R1A5A1008055).
The fourth author's work was supported by JSPS KAKENHI Grant Number JP15K20885.


\end{document}